\documentclass[onefignum,onetabnum]{siamonline190516}

\ifpdf
\hypersetup{
	pdftitle={Stability Domains for Quadratic Reduced-Order Models},
	pdfauthor={B. Kramer}
}
\fi

\usepackage{amsfonts}
\usepackage{graphicx}
\usepackage{color}
\usepackage{enumitem} 
\usepackage{subfigure}
\usepackage{appendix}  
\usepackage{url}
\usepackage{algorithmic}

\newsiamremark{remark}{Remark}

\newcommand{\x}{\mathbf{x}}
\newcommand{\y}{\mathbf{y}}
\newcommand{\z}{\mathbf{z}}
\newcommand{\bmu}{\boldsymbol{\mu}}
\renewcommand{\u}{\mathbf{u}}

\newcommand{\A}{\mathbf{A}}
\newcommand{\B}{\mathbf{B}}
\newcommand{\C}{\mathbf{C}}
\newcommand{\E}{\mathbf{E}}
\newcommand{\F}{\mathbf{F}}
\newcommand{\G}{\mathbf{G}}
\renewcommand{\H}{\mathbf{H}}
\newcommand{\I}{\mathbf{I}}
\newcommand{\J}{\mathbf{J}}
\newcommand{\K}{\mathbf{K}}
\renewcommand{\L}{\mathbf{L}}
\newcommand{\M}{\mathbf{M}}
\newcommand{\N}{\mathbf{N}}
\renewcommand{\P}{\mathbf{P}}
\newcommand{\Q}{\mathbf{Q}}
\newcommand{\R}{\mathbf{R}}
\renewcommand{\S}{\mathbf{S}}
\newcommand{\T}{\mathbf{T}}
\newcommand{\U}{\mathbf{U}}
\newcommand{\V}{\mathbf{V}}
\newcommand{\W}{\mathbf{W}}
\newcommand{\X}{\mathbf{X}}
\newcommand{\bzero}{\mathbf{0}}
\newcommand{\Real}{\mathbb{R}}
\newcommand{\bSigma}{\boldsymbol{\Sigma}}

\headers{Stability Domains for Quadratic-Bilinear Reduced-Order Models}{B. Kramer}

\title{Stability Domains for Quadratic-Bilinear  Reduced-Order Models\thanks{Submitted to the editors \today.
	}}
	
\author{Boris Kramer\thanks{Department of Mechanical and Aerospace Engineering, University of California San Diego, CA 
		(\email{bmkramer@ucsd.edu}, \url{kramer.ucsd.edu}).   }     }

\begin{document}
\maketitle

\begin{abstract}
We propose a computational approach to estimate the stability domain of quadratic-bilinear reduced-order models (ROMs), which are low-dimensional approximations of large-scale dynamical systems. 
For nonlinear ROMs, it is not only important to show that the origin is locally asymptotically stable, but also to quantify if the operative range of the ROM is included in the region of convergence. While accuracy and structure preservation remain the main focus of development for nonlinear ROMs, computational methods that go beyond the existing highly conservative analytical results have been lacking thus far.
In this work, for a given quadratic Lyapunov function, we first derive an analytical estimate of the stability domain, which is rather conservative but can be evaluated efficiently. With the goal to enlarge this estimate, we provide an optimal ellipsoidal estimate of the stability domain by solving a convex optimization problem. This provides us with valuable information about stability properties of the ROM, an important aspect of predictive simulation. 
We do not assume a specific ROM method, so a particular appeal is that the approach is applicable to quadratic-bilinear models obtained via data-driven approaches, where ROM stability properties cannot---per definition---be derived from the full-order model.  
Numerical results for a LQG-balanced ROM of Burgers' equation, a  proper orthogonal decomposition ROM of FitzHugh-Nagumo, and a non-intrusive ROM of Burgers' equation demonstrate the scalability and quantitative advantages of the proposed approach. The optimization-based estimates of the stability domain are found to be up to four orders of magnitude less conservative than analytical estimates.
\end{abstract}	
	
%

\section{Introduction}
Reduced-order modeling provides a mathematical framework for efficient simulation of complex systems, where large-scale nonlinear dynamical systems are approximated on low-dimensional manifolds. High-dimensional systems, arise, e.g., from semi-discretization of partial differential equation (PDE), and their dimension (degrees of freedom) can be in the order of thousands and hundreds of thousands. Reduced-order models (ROMs) of those systems are important in the context of prediction, control, design, and optimization, see, e.g., \cite{antoulas05,BCOW2017morBook,quarteroni2014reduced}. 
%
While the development and analysis of ROMs for linear systems has matured in recent years, ROMs for nonlinear systems---as expected---face significantly different challenges. For general nonlinear systems, proper orthogonal decomposition (POD)~\cite{holmes_lumley_berkooz_1996} and the reduced basis method~\cite{hesthaven2016certified} are the most commonly used model reduction methods. 
%
If we restrict ourselves to the class of quadratic-bilinear (QB) systems, ROM methods have been developed in the intrusive setting (where governing equations are available)~\cite{bennerBreiten2015twoSided,bennergoyal2016QBIRKA,cao2018krylovQB} and in the non-intrusive setting~\cite{gosea2018LoewnerQB,SKHW2020_learning_ROMs_combustor,QKPW2020_lift_and_learn,BGKPW2020_OpInf_nonpoly}, where the model has to be learned from data. 
Notably, QB models are a less restrictive class than the name suggests: many nonlinear dynamical systems can be transformed into QB form via variable transformations and the introduction of auxiliary variables, see~\cite{mccormick1976computability, kerner1981universal,savageau1987recasting,gu2011qlmor,bennerBreiten2015twoSided,KW18nonlinearMORliftingPOD,SKHW2020_learning_ROMs_combustor,guillot2019taylor,QKPW2020_lift_and_learn}. Consequently, reduced-order modeling and analysis developed for QB systems apply to a large class of nonlinear systems.

The characterization of stability domains is important for open and closed loop simulation. In the nonlinear control and dynamical systems community, several approaches exist for estimating the stability domain of (very low-dimensional) polynomial systems, and in particular, quadratic systems. 
The authors in \cite{papachristodoulou2005analysis} use polynomial recasting of nonlinear dynamical systems together with sum-of-squares decomposition to prove stability of nonlinear dynamical systems.  The result is illustrated on numerical examples with at most two DoFs. 
In \cite{najafi2016fast}, a fast sampling approach to estimate the domain of attraction is used. While this direct approach works well in lower dimensions (examples up to third order), scalability issues for higher dimensions remain. 
In \cite{zarei2018arc} it is shown that the arc length function is a maximal Lyapunov function (particularly, it is a Lyapunov function inside the domain of attraction, and approaches infinity on the boundary). A rather expensive computational procedure for approximating the arc length function is proposed, and illustrated on numerical examples  with at most three DoFs. 
The authors in \cite{chesi2007estimatingDA-union-Lyapunov-estimates} parametrize the Lyapunov function used in the stability proof, and construct the domain of attraction as a union of (potentially infinitely many) Lyapunov functions. While results in two and three DoF systems are shown, the authors point out that scalability will be a major issue. 
For quadratic systems, \cite{genesio1990stability-quadratic-systems} present a conservative ellipsoidal estimate of the domain of attraction. The method enlarges the ellipsoidal shape of the Lyapunov function by evaluating the $2^{n-1}$ corners of a polytope in $\Real^n$. 
In \cite{levin1994analytical-method-DA}, an analytical method for estimating the domain of attraction of polynomial systems is proposed via parametrized quadratic Lyapunov functions. The method
, however, provides only conservative analytical estimates of the stability domain. Examples of at most two DoFs are presented. 
The authors in \cite{liu1993global} present a conservative stability analysis for linear systems with quadratic state feedback controllers, where switched feedback is shown to improve the transient system response. The considered systems are only of two DoFs. 
The domain of attraction can virtually have any shape, and while ellipsoidal, circular, polytope solutions are common, they are also conservative. The authors in \cite{pitarch2013closed} propose a method that enlarges an initial guess of the stability domain by using fuzzy polynomials together with sum-of-squares techniques. However, only examples with two DoFs are presented. 
The emergence of sum-of-squares techniques for convex programming has also led to improved algorithms for stability analysis for polynomial systems, mostly with quadratic Lyapunov functions, see \cite{chesi2011domain,anderson2015advances} for more details.
As evidenced by the above literature, stability domain computations have largely focused on very low-dimensional ($\leq4$ DoFs) polynomial systems, and scalability remains a concern.

This work enables quantitative analysis of stability domains for projection-based and fully data-driven ROMs via a computationally tractable optimization algorithm. The analysis of accurate ROMs of high-dimensional systems is a first step towards stability analysis for complex high-dimensional systems.
Our focus is solely on scalable approaches that work well in those (higher, yet not large-scale) ROM dimensions of $\mathcal{O}(10)$.  In particular, we propose to use a convex optimization-based approach from~\cite{tesi1996stability-quadratic-Lyapunov-function} to compute estimates of the stability domain for quadratic-bilinear systems. 
Our contributions are thus threefold: First, we derive a new analytical estimate of the stability domain, which is rather conservative but can be evaluated efficiently, even in the case of high-dimensional systems.
Second, with the goal to enlarge this estimate, we show that our efficient implementation of the  convex optimization-based approach from~\cite{tesi1996stability-quadratic-Lyapunov-function} scales well for $\mathcal{O}(10)$-dimensional ROMs, which goes well beyond the four dimensional systems analyzed in the literature. This allows us to certify ROM simulations with respect to stability domains centered around equilibria. 
Third, we demonstrate the methods' flexibility on several different (projection-based and fully data-driven) ROMs with up to 21 degrees of freedom. This provides us with new insight into the stability characteristics of those ROM techniques. We do not assume a specific ROM method, so a particular appeal is that the approach is applicable to quadratic-bilinear models obtained via data-driven approaches, where ROM stability properties cannot---per definition---be derived from the full-order model.

This paper is organized as follows. Section~\ref{sec:background} presents stability definitions and some necessary background material. Section~\ref{sec:analytical_est} derives an analytical estimate for the stability domain, and Section~\ref{sec:ellipsoidalEst} presents an optimization-based approach to obtain larger, less-conservative estimates. Section~\ref{sec:numerics} illustrates our findings on three test problems with different ROMs for semi-discretized PDE systems. Section~\ref{sec:conclusions} offers conclusions and an outlook to future work.

\section{Background: Stability analysis for quadratic-bilinear systems} \label{sec:background}
We present necessary background material and define the problem under consideration in this paper. Section~\ref{sec:QBsys} introduces the specific model form of quadratic-bilinear (QB) systems, Section~\ref{sec:Equilibrium} discusses equilibrium solutions, and Section~\ref{sec:DA} defines the domain of attraction (also called stability domain, region of attraction, or basin of attraction).

\subsection{Quadratic-bilinear systems}\label{sec:QBsys}
Consider a quadratic-bilinear (QB) system of the form
\begin{align}
\E \dot{\x}  = \A \x + \H (\x \otimes \x) + \sum_{i=1}^n \N_i \x \u_i + \B \u  \label{eq:QBsys}
\end{align}
with state $\x = \x(t) \in \Real^n$, initial condition $\x(0) = \bzero$, input $\u = \u(t) \in \Real^m$, matrices $\E \in \Real^{n\times n}, \ \A \in \Real^{n\times n}, \ \H \in \Real^{n\times n^2}, \ \B\in \Real^{n\times m}$, and  $\N_i \in \Real^{n\times n}$ for $i=1, 2, \ldots, m$.  Here, the matrix $\E$ is assumed non-singular, and the symbol $\otimes$ denotes the standard Kronecker product.  
Without loss of generality, we assume the matrix $\H$ is symmetric in that $\H(\x_1\otimes \x_2) = \H(\x_2\otimes \x_1)$ for any $\x_1, \x_2 \in \Real^n$, which can be enforced without changing the dynamics, see, e.g.,~\cite{bennerBreiten2015twoSided}.	

To ease notation, the material in the following two sections on stability domains is presented for a generic $n$-dimensional system. In Section~\ref{sec:numerics} we introduce the reduced-order modeling context formally, which results in low-order models for which these methods can be applied.

\subsection{Stability of equilibrium solutions and domain of attraction} \label{sec:Equilibrium}
To investigate the stability of equilibrium solutions for QB systems, consider the system
\begin{align}
\E \dot{\x} = \A \x + \H (\x \otimes \x), \label{eq:QB-autonomous}
\end{align}
and invertible matrix $\E$. The above systems includes the autonomous case, i.e, equation~\eqref{eq:QBsys} with $\u\equiv \bzero$ and also the case with state-dependent feedback $\u=\mathcal{K}\x$ (with proper redefinition of the matrices $\A,\H$). It is clear that $\x_e=\bzero$ is an equilibrium. 
Moreover, it is sufficient to study the zero equilibrium of the QB system. To see this, note that all equilibrium solutions are given by $\x_e\neq \bzero$ that satisfy
\begin{align}
\bzero = \A \x_e + \H (\x_e \otimes \x_e).
\end{align}
Introducing $\z = \x - \x_e$ and inserting into equation~\eqref{eq:QB-autonomous} gives
\begin{align}
\E \dot{\z} & = \A \z  + \H (\z \otimes \z) + 2\H (\x_e \otimes \z ) + \underbrace{\A \x_e + \H (\x_e \otimes \x_e)}_{=\bzero} \\
& = [\A + 2\H (\I \otimes \x_e) ]\z + \H (\z \otimes \z).
\end{align}
The resulting system is again quadratic and has zero as an equilibrium. Therefore, without loss of generality we study the zero equilibrium of the QB system, and note that the theory carries over to nonzero equilibria after the model is shifted as shown above.

\subsection{Domain of attraction} \label{sec:DA}
For nonlinear systems, stability of equilibrium point requires a local characterization--whereas for linear systems, stability of the system matrix automatically guarantees global stability. The domain of attraction (DA), also called stability domain, is the set of all initial conditions that result in bounded trajectories that converge to the equilibrium solution. Let $\phi(\x_0,t)$ be a solution to~\eqref{eq:QB-autonomous} for a given initial condition $\x_0$. The domain of attraction for the equilibrium $\x_e$ is defined as
\begin{eqnarray}
\mathcal{A}(\x_e) := \left \{ \x_0: \lim_{t\rightarrow \infty} \phi(\x_0,t)=\x_e \right \}.
\end{eqnarray} 
The domain of attraction can have complicated geometric structure. Computing the domain of attraction analytically is impossible even in very simple cases. A common solution to this problem is to lower-bound the DA by a set $\mathcal{D} \subseteq \mathcal{A}$ of simple geometric shapes, such as ellipsoids or polyhedrons. The task then becomes to make this lower bound as tight as possible. 

We next state the well-known theorem that relates local stability to a suitable Lyapunov function, and characterizes the domain of attraction. 
\begin{theorem} \cite{tesi1996stability-quadratic-Lyapunov-function,chesi2007estimatingDA-union-Lyapunov-estimates}
	If there exists a Lyapunov function $v(\cdot): \Real^n \mapsto \Real^+$ such that 
	$$
	v(\x) > 0, \qquad \dot{v}(\x) <0
	$$	
	for all $\x$ in a neighborhood of zero, then the zero solution is locally asymptotically stable. Moreover, 
	\begin{equation}
	\mathcal{D}(\rho) = \{ \x: v(\x) \leq \rho^2, \ \dot{v}(\x) <0  \}. \label{eq:DA}
	\end{equation}
	is an {\rm{estimate of the domain of attraction}} for the zero equilibrium. 
\end{theorem}
This theorem presents the basis for our computational algorithms following in the next section. Finding the largest $\rho$ that satisfies \eqref{eq:DA} can then be shown to require the solution of a convex optimization problem.

\section{Analytical estimates of the stability domain} \label{sec:analytical_est}
Lyapunov functions are the standard tool for stability analysis of nonlinear systems, and hence for approximating the DA. The choice of the Lyapunov function is non-trivial and influences the DA approximation as the level sets of the Lyapunov function are needed in DA estimates.
In stability analysis for quadratic systems, Lyapunov functions such as polyhedral functions~\cite{amato2008stability-polyhedral-Lyapunov}, polynomials of higher degree~\cite{chesi2007estimatingDA-union-Lyapunov-estimates}, and arc length function approximations~\cite{zarei2018arc} have been used. 
Here, we present an analytical  approach to approximate the domain of attraction. Quadratic Lyapunov functions lead to ellipsoidal estimates of the DA~\cite{genesio1990stability-quadratic-systems,levin1994analytical-method-DA} which are easy to compute, and provide an initial conservative estimate of the DA.

A Lyapunov matrix is a positive definite matrix $\P$ that satisfies 
\begin{equation}
 \A^\top \P \E + \E^\top \P \A + \Q = \bzero \label{eq:LyapMatrix}
\end{equation}
for some positive definite matrix $\Q = \Q_f^\top \Q_f$. Given a Lyapunov matrix, we define a quadratic, nonnegative Lyapunov function via
\begin{equation} \label{eq:lyapunovFunction}
v(\x) = \x^\top \E^\top \P \E \x, 
\end{equation}
with a derivative along trajectories as
\begin{align}
\dot{v}(\x) = \dot{\x}^\top \E^\top \P \x  + \x^\top \P \E \dot{\x}.
\end{align}

\begin{proposition} \label{prop:analyticalEstimate}
Let $\A$ be Hurwitz and $\P$ be a Lyapunov matrix, i.e., a solution to equation~\eqref{eq:LyapMatrix}. Let $v(\x) =  \x^\top \E^\top \P \E\x$ be a Lyapunov function. Then $\x_e=\bzero$ is a locally stable equilibrium and $\mathcal{D}(\rho)\subseteq \mathcal{A}(\bzero)$ with $\rho = \frac{\sigma^2_{\text{min}} (\Q_f)} {2  \Vert \H \Vert_2 \sqrt{\Vert \P \Vert_2}} $ is an estimate of the domain of attraction.
\end{proposition}
\begin{proof}
Since $\P$ is symmetric positive definite, and $\E$ is positive definite, the Lyapunov function $v(\x) =  \x^\top \E^\top \P \E\x$ is positive. We next consider the region where its derivative is negative. We have
\begin{align}
\dot{v}(\x) 
& = \dot{\x}^\top \E^\top \P \E \x  + \x^\top\E^\top \P\E  \dot{\x}, \\
& = [\A \x + \H(\x\otimes \x) ]^\top \P \E \x  + \x^\top \E^\top  \P [\A \x + \H(\x\otimes \x) ], \\
& = \x^\top [\A^\top \P \E + \E^\top  \P \A]\x + (\x^\top \otimes \x^\top ) \H^\top \P \E \x + \x^\top \E^\top \P \H  (\x \otimes \x), \\
& = - \x^\top \Q_f^\top \Q_f \x + 2 \x^\top \E^\top  \P \H (\x \otimes \x), \\
& \leq - \sigma^2_{\text{min}}(\Q_f) \Vert \x \Vert_2^2 + 2 \Vert \x \Vert_2^3 \Vert \E \Vert_2 \Vert \P \Vert_2 \Vert \H\Vert_2.
\end{align}
where $\sigma_{\text{min}}(\Q_f)$ is the smallest non-zero singular value of $\Q_f$. Thus we have that 
\begin{equation}
\dot{v}(\x)<0 \ \ \Leftrightarrow  \ \ \Vert \x \Vert_2 < \frac{\sigma^2_{\text{min}} (\Q_f)} {2 \Vert \E \Vert_2 \Vert \P \Vert_2 \Vert \H \Vert_2}.
\end{equation}
This shows that the zero equilibrium is locally asymptotically stable. To get an estimate of the domain of attraction~\eqref{eq:DA}, we need to find the radial upper bound on the Lyapunov function for all $\x$ where $\dot{v}(\x)<0$. In particular, 
\begin{equation}
v(\x)  
= \x^\top \E^\top \P \E \x 
\leq  \Vert \x \Vert_2^2 \Vert \P \Vert_2 \Vert \E \Vert_2^2
< \frac{\sigma^4_{\text{min}} (\Q_f)} {4 \Vert \P \Vert_2 \Vert \H \Vert^2_2} = \rho^2
\end{equation}
which yields the claimed result after taking a square root.
\end{proof}
\begin{remark}
	Several observations are in order. First, for $\H=\bzero$ we see that $\rho = \infty$, so the stability region is $\Real^n$, which is consistent with the assumption that $\A$ is Hurwitz. 
	Second, note that $\rho$ is always nonzero as $\sigma_{\text{min}}(\Q_f)$ is the smallest non-zero singular value of $\Q_f$, and the Lyapunov matrix is positive definite for the right-hand-side $\Q_f^\top\Q_f$. Therefore $\x_e =\bzero$ is guaranteed to be a stable solution of the QB system as long as $\A$ is Hurwitz. However, the size of the corresponding stability domain (indicated by $\rho$) intuitively is inversely proportional to $\Vert \H \Vert_2$. For weak nonlinear terms, the estimate of the stability domain is larger, and for strong nonlinear terms, this estimate shrinks. 
	Third, as with most analytical  approaches, the result in Proposition~\ref{prop:analyticalEstimate} is rather conservative. Computational approaches can enlarge the estimate of the stability domain, which we discuss next. 
\end{remark}

\section{Estimating the stability domain via optimization} \label{sec:ellipsoidalEst}
We present an optimization-based approach to enlarge the estimate of the stability domain.  For the purpose of this analysis we rewrite the QB system in the form
\begin{equation}
\E \dot{\x}  = \A \x + \sum_{i=1}^{n} x_i \K_i\x,
\end{equation}
with $\x(0)=\bzero$ and where $x_i$ is the $i$th component of the vector $\x$, $\K_i \in \Real^{n\times n}$ are matrices such that $\K_i = \H (:, (i-1)n+1 : in)$ for $i=1, 2, \ldots n$.
We observe that the quadratic part of the right-hand side is invariant under skew-symmetric matrix additions, namely  
$$ 
\sum_{i=1}^{n} x_i \K_i\x = \sum_{i=1}^{n} x_i [\K_i + \S_i] \x
$$ 
for any skew-symmetric matrices $\S_i$. 
Therefore, the nonlinearity can be parametrized with the vector
\begin{equation} \label{eq:dn}
\bmu = \text{vec}([\S_1, \S_2, \ldots, \S_n] ) \in \Real^{d_n}, \qquad d_n = n^2(n-1)/2.
\end{equation}
where vec() stacks the columns of the matrix into a column vector. 
We then define
\begin{equation}
\F(\x,\bmu) = \sum_{i=1}^{n} x_i [\K_i + \S_i] = \sum_{i=1}^{n} x_i \M_i(\bmu), 
\end{equation}
and note that $\F(\x,\bmu_1)\x = \F(\x,\bmu_2)\x$ for $\bmu_1\neq \bmu_2$ but $\F(\x,\bmu_1) \neq \F(\x,\bmu_2)$.
The parametrized dynamics of the quadratic system can be written as 
\begin{equation} \label{eq:QB-parametrized}
	\E \dot{\x}  = \A \x + \F(\x, \bmu)\x.
\end{equation}
This parametrized formulation is the basis for the optimization routine to estimate the stability domain.

In order to compute a less conservative estimate of the stability domain, we define a program to maximize the stability radius $\rho$, as in~\cite{tesi1996stability-quadratic-Lyapunov-function}.  The optimization problem becomes
\begin{equation}  \label{eq:OptProblem}
\begin{aligned}
{\rho^*}^2 & = \inf_{\x \in \Real^n} \ \x^\top \E^\top \P \E \x, \\
\text{s.t.} & \quad \dot{v}(\x) = 0.
\end{aligned}
\end{equation}
The derivative of $v(\cdot)$ along trajectories of equation~\eqref{eq:QB-parametrized} is
\begin{align}
\dot{v}(\x) & = \dot{\x}^\top \E ^\top \P \E \x  + \x^\top\E ^\top  \P \E \dot{\x} \\
& =  [\A \x + \F(\x, \bmu)\x]^\top \P \E \x  + \x^\top \E^\top \P [\A \x + \F(\x, \bmu)\x] \\
& = \x^\top \left [-\Q + \sum_{i=1}^n x_i [\M_i(\bmu)^\top \P \E +  \E^\top \P \M_i(\bmu)  ] \right ] \x
\end{align}
for the positive definite matrix $-(\A^\top \P \E  + \E^\top \P \A) = \Q = \Q_f^\top \Q_f$. Let $\P = \P_f^\top \P_f$,  and define the matrices
\begin{equation*}
\G(\bmu) =  \left [ \M_1(\bmu)^\top \P \E + \E^\top \P \M_1(\bmu) \vert \ldots \vert \M_n(\bmu)^\top \P \E + \E^\top \P \M_n(\bmu)  \right]^\top \in \Real^{n^2\times n}
\end{equation*}
as well as
\begin{equation}
\J(\bmu) = (\P_f^\top \otimes \Q_f^\top)^{-1} \G(\bmu) \Q_f^{-1} \in \Real^{n^2\times n} \label{eq:J}
\end{equation}%
which is \textit{affine linear} in the parameters $\bmu$, a property that is important for optimization later\footnote{If the Lyapunov equation~\eqref{eq:LyapMatrix} is solved with low-rank methods, see e.g., the survey~\cite{benner2013numerical}, then $\P_f^\top \in \Real^{n\times n_r}$ with $n_r \ll n$, in which case the inverse in \eqref{eq:J} becomes a pseudo-inverse.}. 
Taken together, we have 
\begin{equation}
\dot{v}(\x) =  \x^\top \Q_f^\top \left [ -\I_n + (\x^\top \P_f^\top \otimes \I_n)\J(\bmu) \right ] \Q_f \x.
\end{equation}
The next theorem then addresses the solution of the optimization problem~\eqref{eq:OptProblem}. 

\begin{theorem} \cite{tesi1996stability-quadratic-Lyapunov-function} \label{thm:alpha}
For a given Lyapunov matrix \ $\P$, a guaranteed stability domain $\mathcal{D}(\rho) \subseteq \mathcal{A}(\bzero)$ is given through
\begin{equation}
\rho < \rho^* = \frac{1}{\alpha^*},
\end{equation}
where 
\begin{equation}
\alpha^* = \min_{\bmu \in \Real^{d_n}} \Vert \J(\bmu) \Vert_2. \label{eq:min_alpha}
\end{equation}
Moreover, the estimate is optimal, i.e., $\rho^*$ is the maximum achievable value of $\rho$ in equation~\eqref{eq:DA}.
\end{theorem}
We observe that the optimization problem~\eqref{eq:min_alpha} is \textit{convex}, as it involves the two-norm computation of an matrix that is affine linear in $\bmu$. We can therefore employ special convex optimization solvers, which makes this problem formulation appealing, despite the polynomial growth (see \eqref{eq:dn}) of the optimization variables with respect to the dimension of the model.

\section{Numerical results} \label{sec:numerics}
We present three different test cases,  spanning different PDE models, as well as different techniques to obtain the ROMs. This shows the wide applicability of this framework, as it is agnostic to the way the ROM was obtained. 
We consider Burgers' equation in Section~\ref{sec:Burgers-numerics} for which we obtain ROMs through LQG-balanced truncation, and in Section~\ref{sec:FHN-numerics} we use POD projection-based model reduction to obtain ROMs for the FitzHugh-Nagumo system. Section~\ref{sec:BurgersOPINF-numerics} then illustrates our framework on a ROM that was learned purely from data, where that data was generated from yet a  different configuration of Burgers' equation. 

To solve the optimization problem~\eqref{eq:min_alpha}, we use Matlab's \texttt{fmincon} optimizer with relative tolerances \texttt{TolX=0.1}, \texttt{tolFun=0.001}, \texttt{MaxIter=1000}, random initial condition $\bmu_0$, and we bound the entries $\bmu^{(i)}$ of the vector $\bmu$ within the optimization as $-10^4\leq \bmu^{(i)} \leq 10^4$.

\subsection{LQG-balanced ROM for Burgers' equation} \label{sec:Burgers-numerics}

\subsubsection{Burgers' equation and discretization}
We consider the one-dimensional Burgers' equation following the setup in \cite{borggaard2020quadratic}. The PDE model is 
\begin{equation}\label{eq:BurgersPDE1}
\dot{z}(\xi, t) = \epsilon z_{\xi\xi}(\xi, t) - \frac{1}{2}(z^2(\xi, t))_\xi + \sum_{k=1}^m \chi_{[(k-1)/m,k/m]}(\xi) u_k(t)
\end{equation}
for $t>0$, $\xi \in [0,1]$ is the spatial variable and $z(\xi, 0) = z_0(\xi) = 0.5 \sin(2\pi \xi)^2$ for $\xi \in [0,0.5]$ and zero otherwise. The notation $z_{\xi\xi}(\xi,t) := \frac{\partial^2}{\partial \xi^2}z(\xi,t)$ denotes a second order spatial derivative; similarly, $z_\xi(\xi,t)$ denotes a first spatial derivative. Moreover, $z \in H_{\text{per}}^1(0,1)$, which means that the system has periodic boundary conditions. Here, $u_i(t), i=1, 2, \ldots, m$ are the input functions. The function $\chi_{[a,b]}(x)$ denotes the characteristic function on $[a,b]$. Practically, this means that the spatial domain is subdivided into $m$ intervals of equal length, and each control is applied in one of the corresponding intervals.
When discretized with linear finite elements, the $N$-dimensional semi-discretized system has the form
\begin{eqnarray}
\E_N \dot{\x}_N &=&\A_N \x + \H_N (\x_N \otimes \x_N) +\B_N \u \qquad \x_N(0)=\x_{N,0} \\
\y_N &=& \C_N \x_N. 
\end{eqnarray}
The output matrix $\C_N$ produces an observation $\y_N$ of the system. Here, we observe the entire state, so $ \C_N = \I_N$. We choose $m=3$ inputs and the viscosity is set to $\epsilon = 10^{-3}$ to make the nonlinear quadratic term dominant. The model dimension is $N=101$, which is still not tractable for stability domain computation.

\subsubsection{LQG-balanced reduced-order model}
The ROM for Burgers' equation is obtained through the LQG-balancing framework~\cite{jonckheere1983new}, which also allows for finding a proper energy function.  For linear time-invariant systems, the LQG-balancing framework finds a coordinate transformation (and subsequent reduction) so that the LQG solutions to the algebraic Riccati equations of the ROM are equal and diagonal. This yields a ROM with favorable control-theoretic properties. We outline the main steps of this well-known method here, and refer to \cite{jonckheere1983new,antoulas05} for details. First, we compute the solutions to the LQG algebraic Riccati equations based on the linearized system matrices, i.e., 
\begin{equation} \label{eq:ARE}
\begin{aligned}
\A_N\P_N +\P_N \A_N ^{\top} -\P_N \C_N^{\top}\C_N\P_N +\B_N\B_N^{\top} &= \bzero, \\
\A_N^{\top}\Q_N +\Q_N\A_N  -\Q_N\B_N\B_N^{\top}\Q_N +\C_N^{\top}\C_N &= \bzero.
\end{aligned}
\end{equation}
Since $ \P_N$ and $\Q_N$ are symmetric positive definite they can be used to define a Lyapunov function in \eqref{eq:lyapunovFunction}, in a similar way as the symmetric positive definite solutions to the linear Lyapunov equation are used.  The transformation requires the Cholesky factors $\P_N =\R_N\R_N^\top$ and $\Q_N =\L_N\L_N^\top$, and then the singular value decomposition 
$$
\U_N \bSigma_N  \W_N =\L_N^\top\R_N. 
$$
Based on the singular value decay, we choose a truncation order $n\ll N$, and define $\bSigma_n = \bSigma_N(1:n,1:n)$. The LQG-BT projection matrices are $\T_n = \R_N(:,1:n)\W_N(:,1:n) \bSigma_n^{-\frac{1}{2}}$ and $\T_n^{-1} = \bSigma_n^{-\frac{1}{2}} \U_N(:,1:n)^\top \L_N(:,1:n)^\top$, which yields the ROM matrices in $n$-dimensions:
\begin{equation}
\A = \T_n^{-1} \A_N \T_n, \qquad \B = \T_n^{-1}\B_N, \quad \C =\C_N\T_n, \quad \H = \T_n^{-1} \H_N (\T_n \otimes \T_n). 
\end{equation}
Note that the ROM is a quadratic model of dimension $n$ and we only used the matrices of the linearized system to compute the transformations.
Per definition of the LQG balancing transformation, the matrices in reduced dimensions satisfy
\begin{equation} \label{eq:redARE}
\begin{aligned}
\A \bSigma_n + \bSigma_n  \A ^{\top} - \bSigma_n \C^{\top} \C \bSigma_n + \B \B^{\top} &= \bzero, \\
\A^{\top} \bSigma_n + \bSigma_n \A  - \bSigma_n \B \B^{\top} \bSigma_n + \C^{\top} \C &= \bzero.
\end{aligned}
\end{equation}

Figure~\ref{fig:Burgers_LQG}, left, shows the singular values for the matrix $\L_N^\top \R_N$, which are typically used to decide about the truncation of the system. The singular values decay rather slowly at first, but have a steep drop at $n=20$, after which they become machine zero. Therefore, increasing $n$ beyond twenty should not be expected to yield better results ROMs. We compute ROMs of dimensions $n=3, 5, 7, \ldots, 21$, for which we analyze the stability domain in the next section.

\subsubsection{Stability domain computation}
We have seen that  $\bSigma_n$ satisfies the LQG Riccati equations~\eqref{eq:redARE}, i.e., it is symmetric positive definite (trivially so, as a diagonal matrix). Thus, we can choose $\P = \bSigma_n$ to define a Lyapunov function $v(\x) = \x^\top \bSigma_n \x$, and moreover $\P_f = \bSigma_n^{1/2}$ and $\Q_f = \C$.  We then solve the optimization problem~\eqref{eq:OptProblem} to compute the optimal estimate $\rho^*$.
Figure~\ref{fig:Burgers_LQG}, right, shows the stability radii $\rho$ obtained from the analytical estimate in Proposition~\ref{prop:analyticalEstimate}, and the optimal estimate $\rho^*$ from Theorem~\ref{thm:alpha}. 

We observe that the optimized stability domain is several orders of magnitude larger than the conservative analytical estimate from Proposition~\ref{prop:analyticalEstimate}. 
Furthermore, we observe that while the estimate $\rho^*$ initially reduces significantly until $n=13$, the stability region again increases until $n=21$. This hints at the fact that as the ROM increases in fidelity, it also becomes more stable. This observation is not possible by considering only the analytical estimate $\rho$, and is therefore a major appeal to using the optimized estimate. 
The analytic estimate from Proposition~\ref{prop:analyticalEstimate} for the full-order model is $\rho_{\text{FOM}}=9.81\times 10^{-4}$. Note, that the ROM-based analytic estimate appears to converge to zero, therefore not approaching the FOM analytical estimate. We observe that the  optimization-based result for the ROM is much closer to $\rho_{\text{FOM}}$. Due to the high dimension of the state space, the optimization-based estimate $\rho^*_{\text{FOM}}$ was not computationally tractable.

We note that we tried the SMRSOFT toolbox \url{https://www.eee.hku.hk/~chesi/y_smrsoft.htm}, but for $n\in \{2,3\}$ it did not yield a lower bound other than zero, and for higher-order systems it became very cumbersome to implement due to the toolbox not accepting matrix-vector multiplications in right-hand-side. The toolbox seemingly was written for low-order systems, and our work is specifically aimed at high-order systems.

\begin{figure}[ht!]
	\begin{subfigure}{}
		\includegraphics[width=0.48\textwidth]{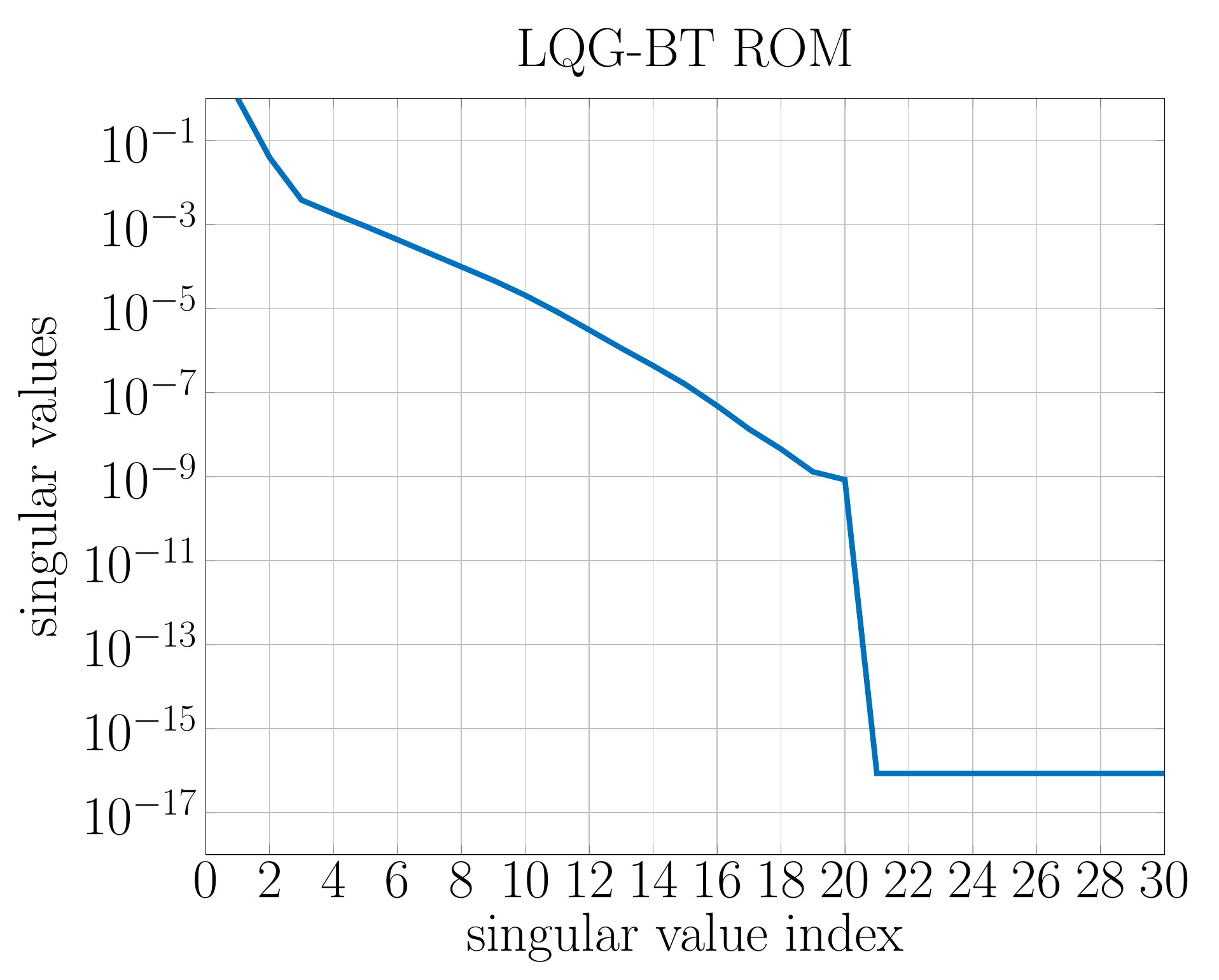}
	\end{subfigure}	
	\begin{subfigure}{}
		\includegraphics[width=0.48\textwidth]{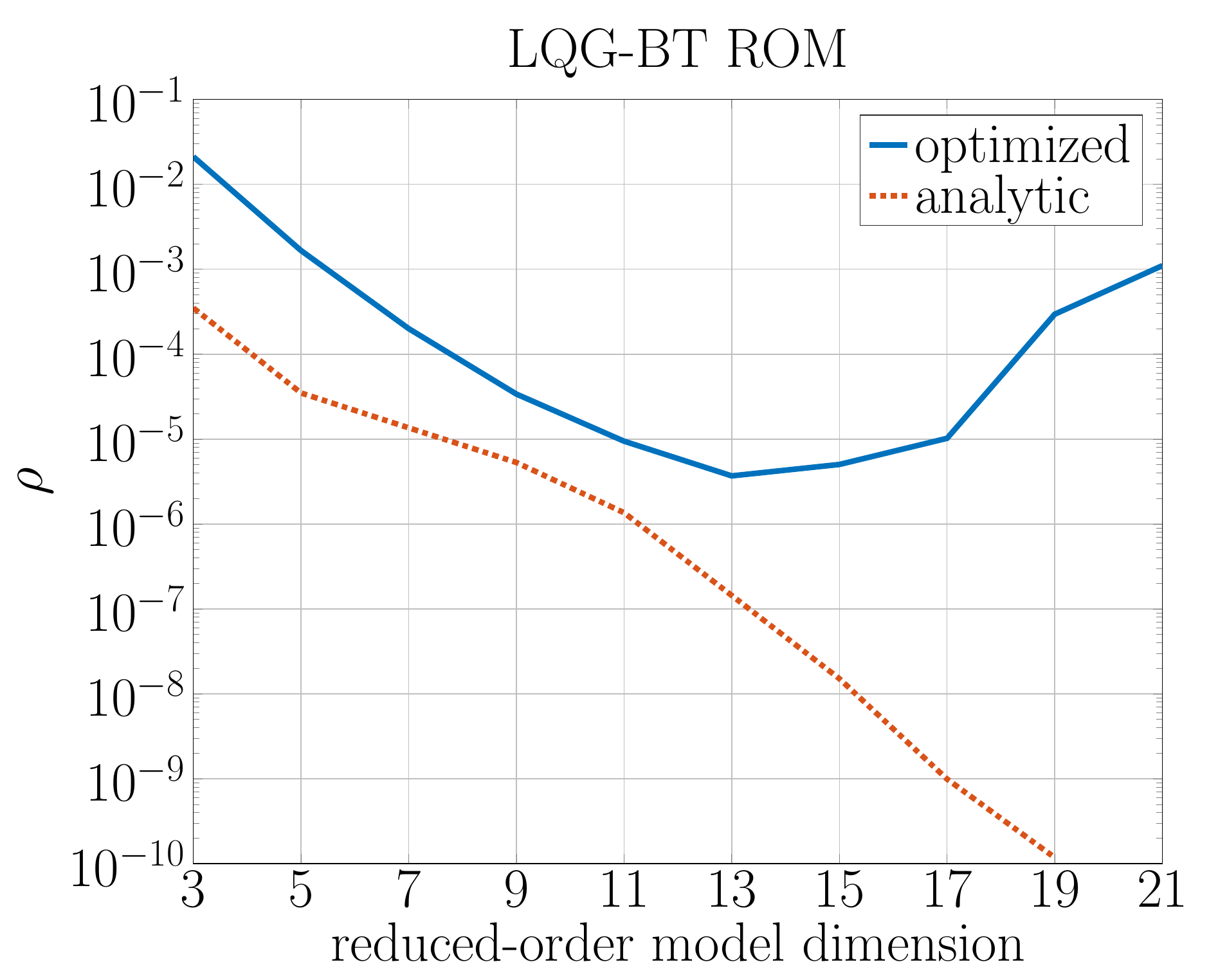}
	\end{subfigure}	
	\caption{Burgers' equation: (Left) Singular values of the matrix $\L_N^\top \R_N$ in LQG balancing; (Right) The value of $\rho$ for the stability domain estimate for the analytic method from Proposition~\ref{prop:analyticalEstimate} and the optimal estimate $\rho^*$ from Theorem~\ref{thm:alpha}. }
	\label{fig:Burgers_LQG}
\end{figure}

\subsection{Proper orthogonal decomposition ROM for FitzHugh-Nagumo equation} \label{sec:FHN-numerics}

\subsubsection{FitzHugh-Nagumo equation and discretization}
This section illustrates our nonlinear model reduction approach on the FitzHugh-Nagumo system, which is a model for the activation and deactivation of a spiking neuron.  The original FitzHugh-Nagumo is a cubic system and here we consider a lifted quadratic bilinear model, see~\cite{bennerBreiten2015twoSided,KW18nonlinearMORliftingPOD} for more details on the model. 
The lifted QB system then reads as
\begin{equation*} 
\begin{aligned}
\epsilon \dot{v}(\xi,t) & = \epsilon^2 v_{\xi\xi}(\xi,t)  - v(\xi,t)^3 + 0.1 v(\xi,t)^2 - 0.1v(\xi,t) - w(\xi,t) + c, \\
\dot{w}(\xi,t) &  = hv(\xi,t) - \gamma w(\xi,t) + c\\
\dot{z}(\xi,t) & =2 [\epsilon^2 vv_{\xi\xi}(\xi,t)  - z^2(\xi,t) + 0.1 z(\xi,t) v(\xi,t) - 0.1 z(\xi,t)  - w(\xi,t) v(\xi,t) + c v(\xi,t)].
\end{aligned}
\end{equation*}
where $\xi \in [0,L]$ is the spatial variable and the time horizon of interest is $t \in [0,t_f]$. The states of the system are voltage $v(\xi,t)$ and recovery of voltage $w(\xi,t)$. 
The initial conditions are specified as $v(\xi,0)=w(\xi,0)=z(\xi,0)=0$ for $\xi \in [0,L]$, and the boundary conditions are $v_\xi(0,t) = u(t)$, where $u(t) = 5\times 10^4 \ t^3 \exp (-15t)$ so the system is excited through the boundary; $v_\xi(L,t) =0$ for $t\geq 0$; $z_\xi(L,t) = 2 v(L,t), z_s(0,t) = 2 v(0,t)u(t)$.
In the problem setup we consider, the parameters are given by $L=0.1$, $c=0.05$, $\gamma =2$, $h=0.5$, and $\epsilon = 0.015$.
The PDE model is semi-discretized in space by using finite differences, resulting in the finite-dimensional QB system
\begin{equation*}
\E_N \dot{\x}_N = \A_N \x_N + \B_N \u + \H_N (\x_N \otimes \x_N) + \sum_{k=1}^2 \N_{N,k} \x_N \u_k ,
\end{equation*}
where $\E_N=\epsilon \I_{N}$ is diagonal, $\A_N,\N_{N,1}, \N_{N,2} \in \mathbb{R}^{N\times N}$ and $\H_N \in \mathbb{R}^{N\times N^2}$. The input matrix is $\B_N \in \mathbb{R}^{N\times 2}$, with the second column of $\B_N$ being copies of $c$ (the constant in the FHN PDE) and the first column of $\B_N$ having a 1 at the first entry. Thus, the input $\u = [u(t), 1]$. 
Here, each variable is discretized with $200$ degrees of freedom, i.e., the overall dimension of the QB model is $N=600$.

\subsubsection{Proper orthogonal decomposition reduced-order model}
We generate a ROM via the method of proper orthogonal decomposition~\cite{holmes_lumley_berkooz_1996} and follow directly the implementation in~\cite{KW18nonlinearMORliftingPOD}. 
We simulate the system for $t_f = 12s$ with Matlab's \texttt{ode15s} solver with \texttt{'RelTol',1e-8,'AbsTol',1e-10} tolerances, and collect simulated data every 0.1s, for a total of 120 snapshots. For each of the three discretized variables we compute a POD basis of order $n$, where we use the same $n$ for each variable. The POD basis is the optimal basis to represent the snapshot set in the $\ell_2$ sense. The projection matrix is then assembled as a block-diagonal matrix with the POD basis for each variable as blocks.  %
Thus, the ROM is the $3n$ dimensional ROM
\begin{equation*}
\E \dot{\x} = \A \x + \B \u + \H (\x \otimes \x) + \sum_{k=1}^2 \N_{k} \x \u_k.
\end{equation*}
Figure~\ref{fig:FHN_POD}, left, shows the decay of the singular values of the POD snapshot matrices for each variable. We observe that until $n=12$ the singular values drop significantly, after which a slow decay appears. Therefore, for $n\leq12$, increasing $n$ should lead to improved ROM accuracy.

\subsubsection{Stability domain computation}
For the Lyapunov matrix, we choose $\Q = \I_n$ and compute the corresponding solution $\P$ of the Lyapunov equation~\eqref{eq:LyapMatrix}. We again compute the stability radii $\rho$ obtained from the analytical estimate in Proposition~\ref{prop:analyticalEstimate}, and the optimal estimate $\rho^*$ from Theorem~\eqref{thm:alpha}. 
Figure~\ref{fig:FHN_POD}, right, shows the obtained results, where the horizontal axis plots the reduced dimension $3n$, so for the ROM of dimension 21, each variable was approximated with seven POD modes. Similar to the previous example, we observe that the optimal computed estimate is roughly three orders of magnitude larger than the analytic one. Moreover, this time the size of the stability domain seems to stay approximately the same after six modes are used. 
We also compute the analytic estimate from Proposition~\ref{prop:analyticalEstimate} for the full-order model, and obtain $\rho_{\text{FOM}}=2.27\times 10^{-6}$. We observe that the ROM stability radius approximates the FOM stability radius well, within the same order of magnitude. This, together with a measure of accuracy of predictions could give confidence in having a good reduced-order model strategy.  Due to the high dimension of the state space, the optimization-based estimate $\rho^{*}_{\text{FOM}}$ was not computationally tractable.

\begin{figure}[!ht]
	\begin{subfigure}{}
		\includegraphics[width=0.48\textwidth]{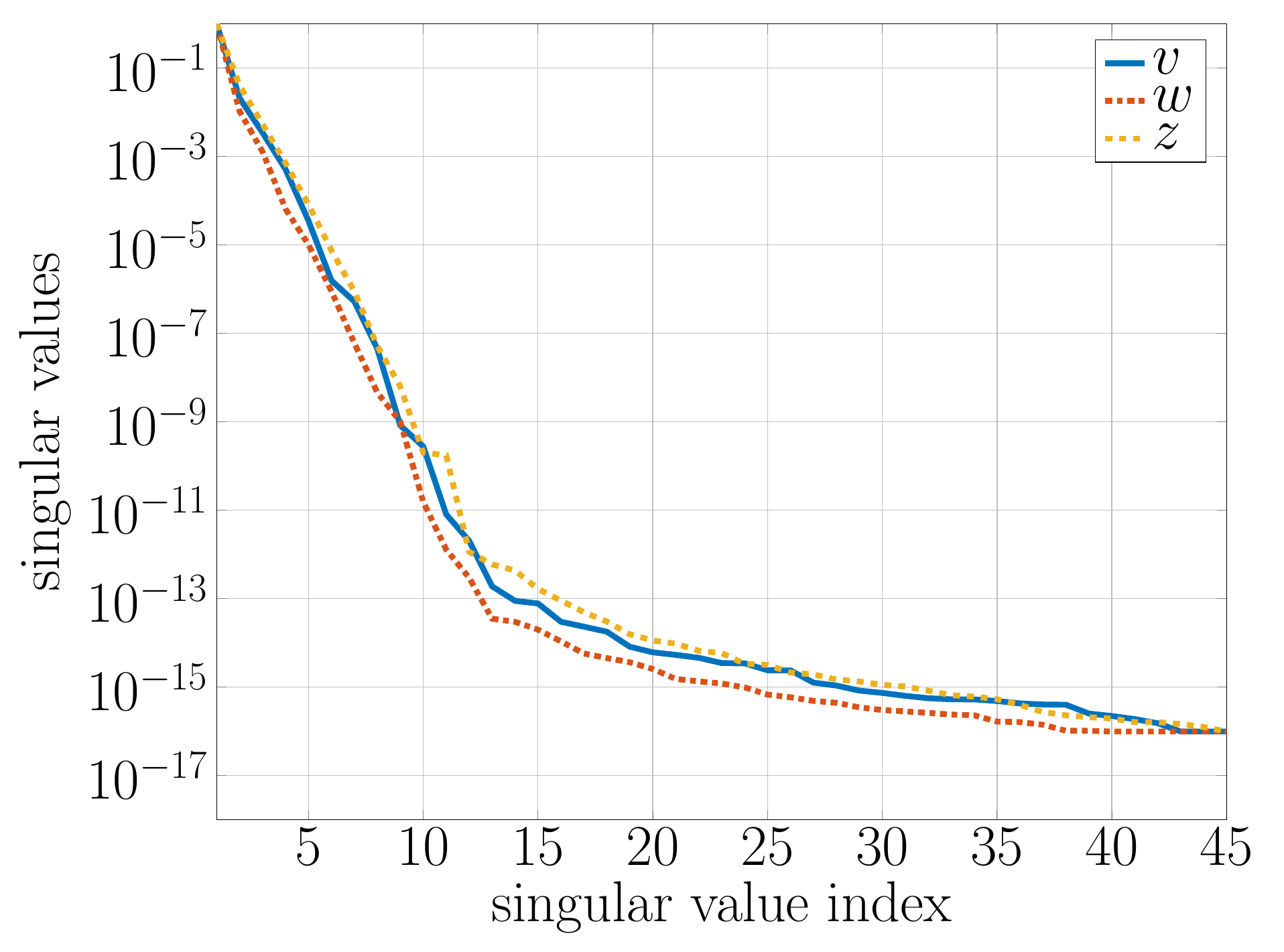}
	\end{subfigure}	
	\begin{subfigure}{}
		\includegraphics[width=0.48\textwidth]{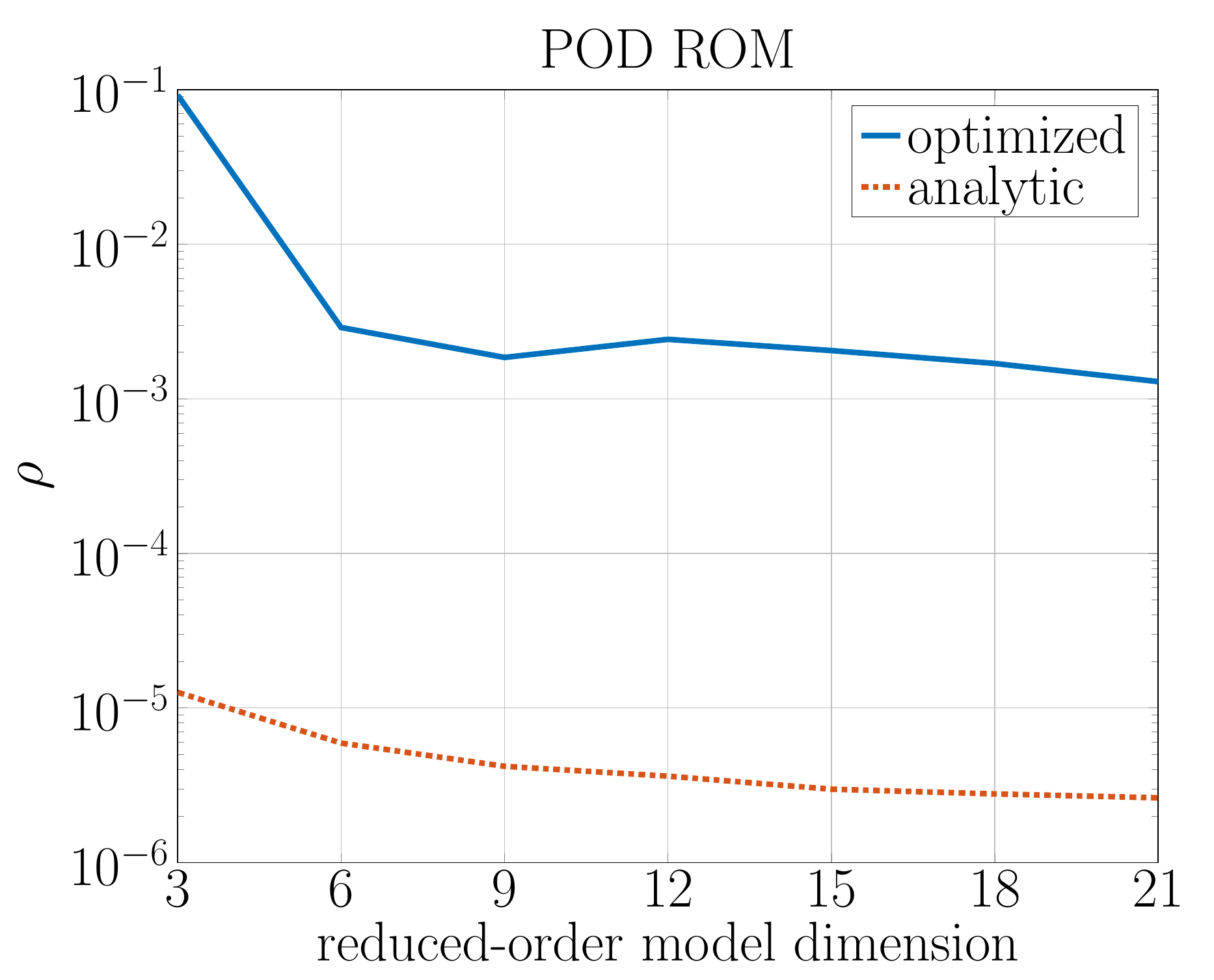}
	\end{subfigure}	
	\caption{FitzHugh-Nagumo equation: (Left) Singular values of the POD snapshot matrix for each variable. (Right) The value of $\rho$ for the stability domain estimate for the analytic method from Proposition~\ref{prop:analyticalEstimate} and the optimal estimate from Theorem~\ref{thm:alpha}.}
	\label{fig:FHN_POD}
\end{figure}

\subsection{Non-intrusive ROM for Burgers' equation} \label{sec:BurgersOPINF-numerics}
In this example, we demonstrate one of the major appeals of this method, namely that we can apply it to non-intrusive ROMs. The model presented here is learned from data of Burgers' equation using the operator inference method following the setup in \cite[Sec 4.2]{peherstorfer2016data}. This setup deviates from Section~\ref{sec:Burgers-numerics} in the discretization parameters and implementation of the control term. 

\subsubsection{Learned reduced-order model}
We are given a simulator for the following Burgers' equation PDE:
\begin{equation}\label{eq:BurgersPDE2}
\dot{z}(\xi, t) = \epsilon z_{\xi\xi}(\xi, t) - \frac{1}{2}(z^2(\xi, t))_\xi
\end{equation}
for $t>0$, $z(\cdot, 0) = 0$, and the input enters through the Dirichlet boundary condition, $z(t,0) = u(t)$ and $z(t,1) = -u(t)$. 
We generate non-intrusive ROMs via the operator inference\footnote{For a Matlab code for this model, and the operator inference approach, see \url{https://github.com/elizqian/operator-inference}. A scalable Python implementation can be found at \url{https://github.com/Willcox-Research-Group/rom-operator-inference-Python3}.} approach from~\cite{peherstorfer2016data}.
The data are generated from a finite difference solver with equidistant grid of $N=128$. The viscosity is set to $\epsilon = 10^{-1}$ and the model is simulated for $t\in [0,t_f]$ and $t_f=1$s. The one-dimensional input $u(t)$ is generated via Matlab's \texttt{rand} command with setting \texttt{rng default} for reproducibility.  We collect snapshots of the state every $dt = 10^{-4}$ steps and store the snapshots and the corresponding inputs as
\begin{equation*}
\X_N = [\x_N^{(0)} \ \dots \x_N^{(K)}] \in \mathbb{R}^{N \times K}, \qquad \mathbf{\U}= [\u_0,\ \dots, \ \u_K] \in \mathbb{R}^{m \times K},
\end{equation*}
where $\x_N^{(i)}  = \x_N(t_i)$ with $0 = t_0 < t_1< \dots < t_K = t_f$. We next compute a low-dimensional POD subspace in which to optimally represent the snapshot data, i.e., we compute the singular value decomposition of the snapshot matrix as
$$
\X_N = \V_N \bSigma_N \W_N^{\top},
$$
where $\V_N\in \mathbb{R}^{N \times K}$, $\bSigma \in \mathbb{R}^{K \times K}$ and $\W_N \in \mathbb{R}^{K \times K}$. We then obtain the $n\ll N$-dimensional POD basis as $\V_n = \V_N(:,1:n)$. 
The high-dimensional snapshot data are projected onto the POD subspace spanned by the columns of $\V_n$, which yields the projected data
\begin{equation*}
\X = \V^{\top}_n \X_N = [\x^{(0)} \quad \dots \quad \x^{(K)}] \in \mathbb{R}^{n \times K},
\qquad
\dot{\X} = [ \dot{\x}^{(0)} \quad \dot{\x}^{(1)} \quad \dots \quad  \dot{\x}^{(K)}]\in \mathbb{R}^{n \times K}.
\end{equation*}
The columns of the time-derivative data matrix $\dot{\X}$ are computed with a fourth-order implicit Runge-Kutta backward differencing method.

In order to learn the ROM, the operator inference framework solves a least squares problem to find the reduced operators that yield the ROM that best matches the projected snapshot data in a minimum residual sense. To learn the Burgers' ROM, we solve
\begin{equation*}
\min_{ \A \in \mathbb{R}^{n \times n} , \H \in \mathbb{R}^{n \times n^2}, \B  \in \mathbb{R}^{n \times m}}
\left \Vert   \X_N^{\top}\A^{\top} + (\X_N \otimes \X_N)^{\top} \H^{\top}  + \U^{\top}\B^{\top} - \dot{\X}_N^{\top} \right \Vert^2_2.
\end{equation*}
This allows us to compute the ROM operators  $\A$, $\H$, and $\B$ without needing explicit access to the original high-dimensional operators $\A_N$, $\H_N$, $\B_N$, which constitutes a fully non-intrusive method.
Finally, the learned ROM takes the form
\begin{equation} \label{eq:quadratic_system_reduced}
\dot{\x} = \A \x+ \H (\x \otimes \x) + \B \u.
\end{equation}
Figure~\ref{fig:Burgers_OPINF}, left, shows the relative state reconstruction error $\Vert \X_N - \V_N \X\Vert / \Vert \X_N \Vert$ for the operator inference ROM. For $n=1,2,3$ we did not obtain stable models. The state reconstruction error decays monotonically, and operator inference produces accurate ROM simulations.

\subsubsection{Stability domain computation}
For the Lyapunov matrix, we choose $\Q = \I_n$ and compute the corresponding solution $\P$ of the Lyapunov equation~\eqref{eq:LyapMatrix}. We again compute the stability radii $\rho$ obtained from the analytical estimate in Proposition~\ref{prop:analyticalEstimate}, and the optimal estimate $\rho^*$ from Theorem~\eqref{thm:alpha}.\footnote{Since this is a non-intrusive setting, we cannot compute the analytic stability radius for the full-order model, as we only have model data, but not the FOM operators.}
Figure~\ref{fig:Burgers_OPINF}, right, shows the numerical result. First, we note that for $n>12$, the ROM simulations continued to increase in accuracy, but the $\A$ matrix had few eigenvalues that were positive (in the order $\mathcal{O}(10^{-8})$), so the approach to compute a quadratic Lyapunov function via solution of the Lyapunov matrix \ref{eq:LyapMatrix} is not applicable. Other approaches that work in this setting are needed, which we will discuss as part of future work. 
Similar to the previous examples, we observe from Figure~\ref{fig:Burgers_OPINF}, right, that the analytical estimate $\rho$  is two to four orders of magnitude smaller than the estimate $\rho^*$ computed via optimization. Moreover, we observe that while $\rho, \rho^*$ initially decrease, after $n=8$ they stagnate and then increase, indicating improved stability properties of the larger-dimensional ROMs. 

\begin{figure}[ht!]
	\begin{subfigure}{}
		\includegraphics[width=0.48\textwidth]{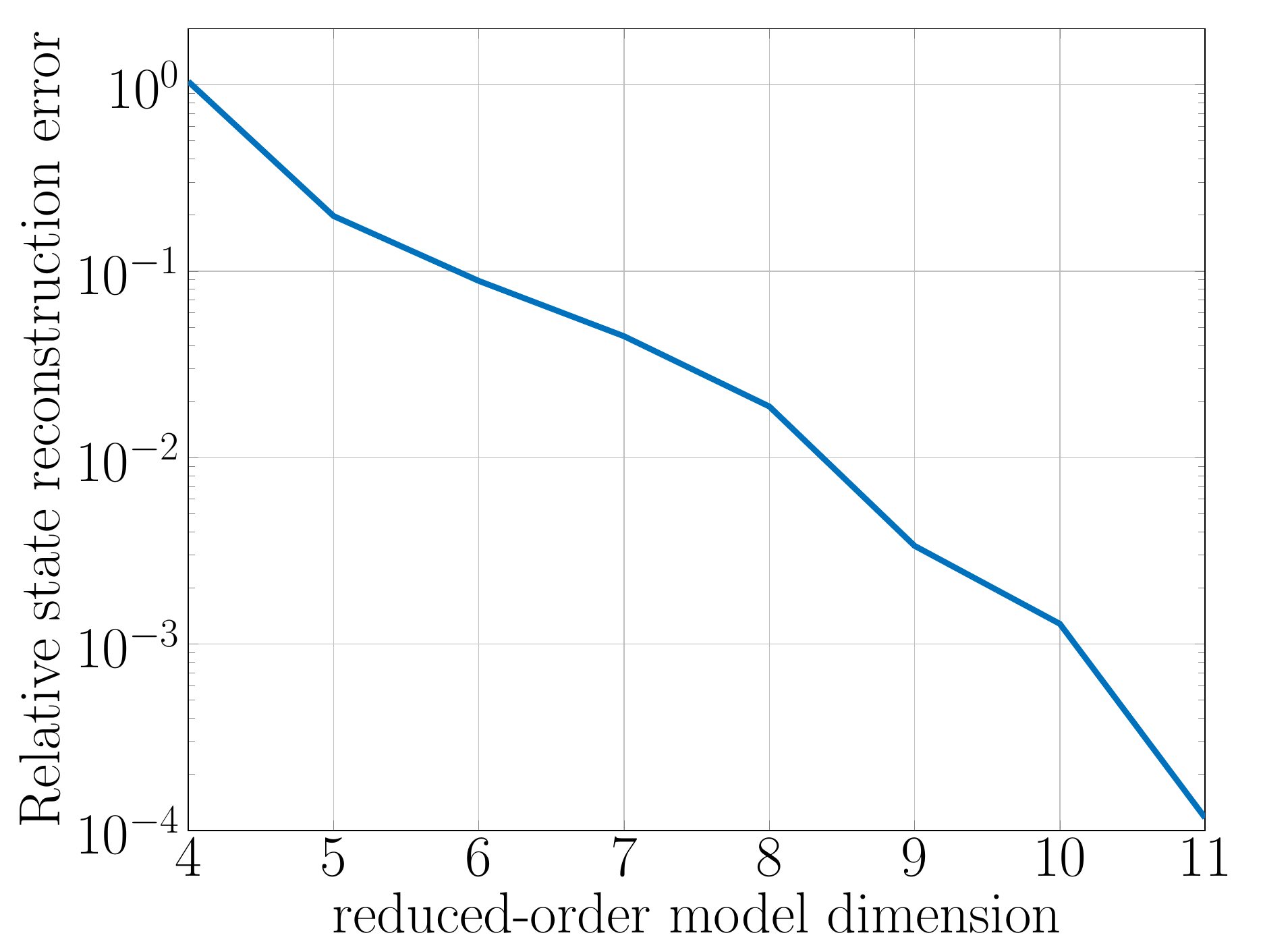}
	\end{subfigure}	
	\begin{subfigure}{}
		\includegraphics[width=0.48\textwidth]{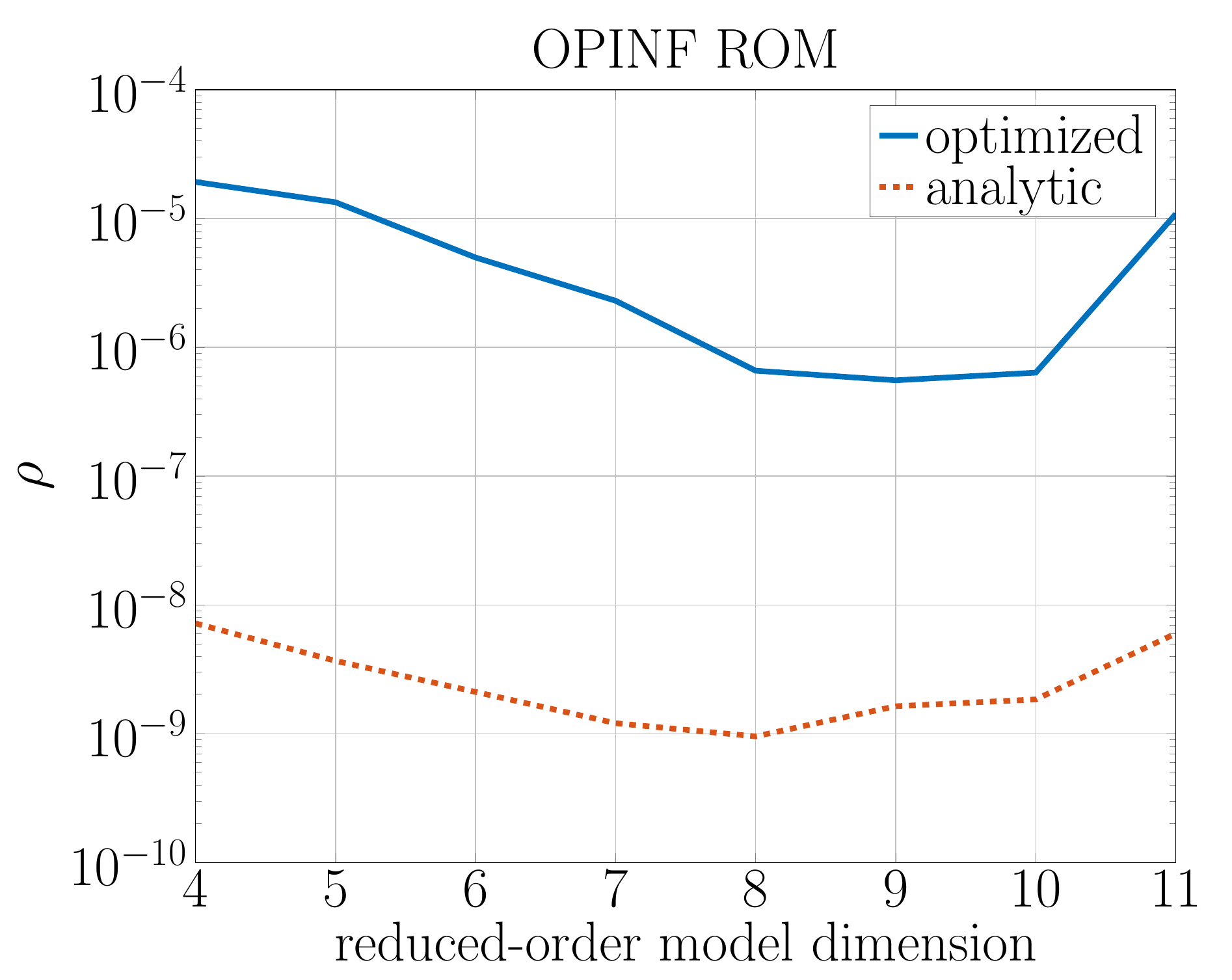}
	\end{subfigure}	
	\caption{Learned ROM for Burgers' equation: (Left) Relative state reconstruction error $\Vert \X_N - \V_N \X\Vert / \Vert \X_N \Vert$. (Right) The value of $\rho$ for the stability domain estimate for the analytic method from Proposition~\ref{prop:analyticalEstimate} and the optimal estimate $\rho^*$ from Theorem~\ref{thm:alpha}.}
	\label{fig:Burgers_OPINF}
\end{figure}

\section{Conclusions and future directions}	\label{sec:conclusions}
For nonlinear reduced-order models (ROMs), computing the stability domain of equilibrium points is important for both open and closed-loop applications. We presented a framework to compute the optimal stability domain estimates for quadratic-bilinear ROMs for a given quadratic Lyapunov function. Quadratic-bilinear ROMs represent a large class of nonlinear systems, as many of those systems can be recast into QB form via variable transformations and the addition of extra variables. 
Our numerical findings on three different ROM test problems show that the classical analytical estimates of the stability domain are overly conservative---up to four orders of magnitude---compared to the stability radii computed via the suggested convex optimization problem.  The numerical results also demonstrate various ways to pick quadratic Lyapunov functions, which can be informed by the model reduction process itself. 
This work motivates several directions of future research. First, as seen in Section~\ref{sec:BurgersOPINF-numerics}, this approach has its limitations when the linear system matrix is unstable, since the Lyapunov equation does not have a solution. Alternative approaches in this case are needed that can take into account the quadratic influence on the system. 
Moreover, Section~\ref{sec:BurgersOPINF-numerics} suggests an interesting direction of future research. The operator inference problem is unconstrained, however, the stability domain computation could be integrated as an additional constraint for the learning problem, yielding learned ROMs with favorable stability properties. 
Lastly, in the projection-based setting, theoretical results relating the ROM stability radius to the high-fidelity stability radius would be desirable. For large-scale systems, a proper choice of norm would like need to be made to avoid degeneration of the $l_2$ norm.

\bibliographystyle{abbrv}
\bibliography{stability}

\begin{thebibliography}{10}

\bibitem{amato2008stability-polyhedral-Lyapunov}
F.~Amato, F.~Calabrese, C.~Cosentino, and A.~Merola.
\newblock Stability analysis of nonlinear quadratic systems via polyhedral
  {L}yapunov functions.
\newblock {\em Automatica}, 47(3):614--617, 2011.

\bibitem{anderson2015advances}
J.~Anderson and A.~Papachristodoulou.
\newblock Advances in computational {L}yapunov analysis using sum-of-squares
  programming.
\newblock {\em Discrete \& Continuous Dynamical Systems-B}, 20(8):2361, 2015.

\bibitem{antoulas05}
A.~C. Antoulas.
\newblock {\em Approximation of Large-Scale Dynamical Systems}.
\newblock Advances in Design and Control. Society for Industrial and Applied
  Mathematics, Philadelphia, PA, USA, 2005.

\bibitem{bennerBreiten2015twoSided}
P.~Benner and T.~Breiten.
\newblock Two-sided projection methods for nonlinear model order reduction.
\newblock {\em SIAM Journal on Scientific Computing}, 37(2):B239--B260, 2015.

\bibitem{BCOW2017morBook}
P.~Benner, A.~Cohen, M.~Ohlberger, and K.~Willcox, editors.
\newblock {\em Model Reduction and Approximation: Theory and Algorithms}.
\newblock Computational Science \& Engineering. SIAM Publications,
  Philadelphia, PA, 2017.

\bibitem{bennergoyal2016QBIRKA}
P.~Benner, P.~Goyal, and S.~Gugercin.
\newblock H2-quasi-optimal model order reduction for quadratic-bilinear control
  systems.
\newblock {\em SIAM Journal on Matrix Analysis and Applications},
  39(2):983--1032, 2018.

\bibitem{BGKPW2020_OpInf_nonpoly}
P.~Benner, P.~Goyal, B.~Kramer, P.~Peherstorfer, and K.~Willcox.
\newblock Operator inference for non-intrusive model reduction of systems with
  non-polynomial nonlinear terms.
\newblock {\em Computer Methods in Applied Mechanics and Engineering},
  372:113433, 2020.

\bibitem{benner2013numerical}
P.~Benner and J.~Saak.
\newblock Numerical solution of large and sparse continuous time algebraic
  matrix riccati and lyapunov equations: a state of the art survey.
\newblock {\em GAMM-Mitteilungen}, 36(1):32--52, 2013.

\bibitem{borggaard2020quadratic}
J.~Borggaard and L.~Zietsman.
\newblock The {Q}uadratic-{Q}uadratic {R}egulator problem: Approximating
  feedback controls for quadratic-in-state nonlinear systems.
\newblock In {\em 2020 American Control Conference (ACC)}, pages 818--823.
  IEEE, 2020.

\bibitem{cao2018krylovQB}
X.~Cao, J.~Maubach, S.~Weiland, and W.~Schilders.
\newblock A novel {K}rylov method for model order reduction of quadratic
  bilinear systems.
\newblock In {\em 2018 IEEE Conference on Decision and Control (CDC)}, pages
  3217--3222. IEEE, 2018.

\bibitem{chesi2007estimatingDA-union-Lyapunov-estimates}
G.~Chesi.
\newblock Estimating the domain of attraction via union of continuous families
  of {L}yapunov estimates.
\newblock {\em Systems \& Control Letters}, 56(4):326--333, 2007.

\bibitem{chesi2011domain}
G.~Chesi.
\newblock {\em Domain of attraction: analysis and control via {SOS}
  programming}, volume 415.
\newblock Springer Science \& Business Media, 2011.

\bibitem{genesio1990stability-quadratic-systems}
R.~Genesio and A.~Tesi.
\newblock Stability analysis of quadratic systems.
\newblock In {\em Nonlinear Control Systems Design 1989}, pages 195--199.
  Elsevier, 1990.

\bibitem{gosea2018LoewnerQB}
I.~V. Gosea and A.~C. Antoulas.
\newblock Data-driven model order reduction of quadratic-bilinear systems.
\newblock {\em Numerical Linear Algebra with Applications}, 25(6):e2200, 2018.

\bibitem{gu2011qlmor}
C.~Gu.
\newblock {QLMOR}: A projection-based nonlinear model order reduction approach
  using quadratic-linear representation of nonlinear systems.
\newblock {\em IEEE Transactions on Computer-Aided Design of Integrated
  Circuits and Systems}, 30(9):1307--1320, 2011.

\bibitem{guillot2019taylor}
L.~Guillot, B.~Cochelin, and C.~Vergez.
\newblock A {T}aylor series-based continuation method for solutions of
  dynamical systems.
\newblock {\em Nonlinear Dynamics}, 98(4):2827--2845, 2019.

\bibitem{hesthaven2016certified}
J.~S. Hesthaven, G.~Rozza, B.~Stamm, et~al.
\newblock {\em Certified reduced basis methods for parametrized partial
  differential equations}, volume 590.
\newblock Springer.

\bibitem{holmes_lumley_berkooz_1996}
P.~Holmes, J.~L. Lumley, and G.~Berkooz.
\newblock {\em Turbulence, Coherent Structures, Dynamical Systems and
  Symmetry}.
\newblock Cambridge Monographs on Mechanics. Cambridge University Press, 1996.

\bibitem{jonckheere1983new}
E.~Jonckheere and L.~Silverman.
\newblock A new set of invariants for linear systems--application to reduced
  order compensator design.
\newblock {\em IEEE {T}ransactions on {A}utomatic Control}, 28(10):953--964,
  1983.

\bibitem{kerner1981universal}
E.~H. Kerner.
\newblock Universal formats for nonlinear ordinary differential systems.
\newblock {\em Journal of Mathematical Physics}, 22(7):1366--1371, 1981.

\bibitem{KW18nonlinearMORliftingPOD}
B.~Kramer and K.~Willcox.
\newblock Nonlinear model order reduction via lifting transformations and
  proper orthogonal decomposition.
\newblock {\em AIAA Journal}, 57(6):2297--2307, 2019.

\bibitem{levin1994analytical-method-DA}
A.~Levin.
\newblock An analytical method of estimating the domain of attraction for
  polynomial differential equations.
\newblock {\em IEEE Transactions on Automatic Control}, 39(12):2471--2475,
  1994.

\bibitem{liu1993global}
J.-S. Liu and S.-L. Chen.
\newblock On global stability of quadratic state feedback controlled linear
  systems.
\newblock {\em Systems \& control letters}, 21(5):371--379, 1993.

\bibitem{mccormick1976computability}
G.~P. McCormick.
\newblock Computability of global solutions to factorable nonconvex programs:
  Part {I}--convex underestimating problems.
\newblock {\em Mathematical Programming}, 10(1):147--175, 1976.

\bibitem{najafi2016fast}
E.~Najafi, R.~Babu{\v{s}}ka, and G.~A. Lopes.
\newblock A fast sampling method for estimating the domain of attraction.
\newblock {\em Nonlinear dynamics}, 86(2):823--834, 2016.

\bibitem{papachristodoulou2005analysis}
A.~Papachristodoulou and S.~Prajna.
\newblock Analysis of non-polynomial systems using the sum of squares
  decomposition.
\newblock In {\em Positive polynomials in control}, pages 23--43. Springer,
  2005.

\bibitem{peherstorfer2016data}
B.~Peherstorfer and K.~Willcox.
\newblock Data-driven operator inference for nonintrusive projection-based
  model reduction.
\newblock {\em Computer Methods in Applied Mechanics and Engineering},
  306:196--215, 2016.

\bibitem{pitarch2013closed}
J.~L. Pitarch, A.~Sala, and C.~V. Arino.
\newblock Closed-form estimates of the domain of attraction for nonlinear
  systems via fuzzy-polynomial models.
\newblock {\em IEEE transactions on cybernetics}, 44(4):526--538, 2013.

\bibitem{QKPW2020_lift_and_learn}
E.~Qian, B.~Kramer, B.~Peherstorfer, and K.~Willcox.
\newblock Lift \& {L}earn: Physics-informed machine learning for large-scale
  nonlinear dynamical systems.
\newblock {\em Physica {D}: {N}onlinear {P}henomena}, 406:132401, 2020.

\bibitem{quarteroni2014reduced}
A.~Quarteroni and G.~Rozza, editors.
\newblock {\em Reduced Order Methods for Modeling and Computational Reduction},
  volume~9 of {\em MS \& A}.
\newblock Springer, Berlin, Heidelberg, New York, 2014.

\bibitem{savageau1987recasting}
M.~A. Savageau and E.~O. Voit.
\newblock Recasting nonlinear differential equations as {S}-systems: a
  canonical nonlinear form.
\newblock {\em Mathematical Biosciences}, 87(1):83--115, 1987.

\bibitem{SKHW2020_learning_ROMs_combustor}
R.~Swischuk, B.~Kramer, C.~Huang, and K.~Willcox.
\newblock Learning physics-based reduced-order models for a single-injector
  combustion process.
\newblock {\em AIAA Journal}, 58:6:2658--2672, 2020.

\bibitem{tesi1996stability-quadratic-Lyapunov-function}
A.~Tesi, F.~Villoresi, and R.~Genesio.
\newblock On the stability domain estimation via a quadratic {L}yapunov
  function: convexity and optimality properties for polynomial systems.
\newblock {\em IEEE Transactions on Automatic Control}, 41(11):1650--1657,
  1996.

\bibitem{zarei2018arc}
M.~Zarei, A.~Kalhor, and D.~Brake.
\newblock Arc length based maximal {L}yapunov functions and domains of
  attraction estimation for polynomial nonlinear systems.
\newblock {\em Automatica}, 90:164--171, 2018.

\end{thebibliography}

\end{document}